\documentclass[11pt]{article}

\usepackage[a4paper, left=23mm,right=23mm, top=30mm, bottom=30mm]{geometry}
\usepackage{amsmath,amssymb,amsthm,mathtools}
\usepackage{hyperref}
\usepackage[nameinlink,noabbrev,capitalise]{cleveref}
\usepackage{enumitem}
\usepackage[dvipsnames]{xcolor}
\usepackage{graphicx}

\newtheorem{theorem}{Theorem}
\newtheorem{lemma}[theorem]{Lemma}
\newtheorem{corollary}[theorem]{Corollary}
\theoremstyle{definition}

\DeclareMathOperator{\mad}{mad}

\newcommand{\G}{\mathcal{G}}
\newcommand{\A}{\mathcal{A}}

\AddToHook{env/lemma/begin}{\crefalias{theorem}{lemma}}
\AddToHook{env/corollary/begin}{\crefalias{theorem}{corollary}}

\begin{document}

\title{Proof of the mad conjecture and its coloring applications}

\author{
Andrzej Grzesik\thanks{Faculty of Mathematics and Computer Science, Jagiellonian University, {\L}ojasiewicza 6, 30-348 Krak\'{o}w, Poland. E-mail: {\tt andrzej.grzesik@uj.edu.pl}.
Supported by the National Science Centre grant 2021/42/E/ST1/00193.}
\and
Lenka Kopfov\'a\thanks{Institute of Science and Technology Austria (ISTA). Am Campus 1, 3400 Klosterneuburg, Austria. E-mail: {\tt lenka.kopfova@ist.ac.at}. Supported by ERC Starting Grant ``RANDSTRUCT'' No. 101076777.
}
\and
Gaurav Kucheriya\thanks{Department of Applied Mathematics, Charles University, Czechia.
E-mail: {\tt gaurav@kam.mff.cuni.cz}.
Supported by GAČR grant 25-17377S and GAUK project 92125.}
\and
Binlong Li\thanks{School of Mathematics and Statistics, 
Xi’an-Budapest Joint Research Center for Combinatorics,
Northwestern Polytechnical University,
Xi’an, Shaanxi 710129, China.
E-mail: {\tt binlongli@nwpu.edu.cn}. Supported by the National Natural Science Foundation of China
(No. 12671417).}
\and
Magdalena Prorok\thanks{AGH University of Krakow, al. Mickiewicza 30, 30-059 Krak\'ow, Poland. 
E-mail: {\tt prorok@agh.edu.pl}.
Supported by the National Science Centre grant 2021/42/E/ST1/00193.}
}

\date{}

\maketitle

\begin{abstract}
For a finite graph $G$, the maximum average degree $\mad(G)$ is the largest average degree of a nonempty subgraph of $G$. Hendrey, Norin and Wood asked whether this parameter is partitionable, that is, whether for all positive reals $a,b$ every graph $G$ with $\mad(G)<a+b$ admits a vertex partition $V(G)=A\cup B$ with $\mad(G[A])<a$ and $\mad(G[B])<b$. In this note, we answer this question affirmatively. Moreover, we generalize this to an arbitrary number of partition classes.

The clustered chromatic number $\chi_\star(\G)$ of a graph class $\G$ is the minimum integer $k$ such that, for some integer $c$, every graph in $\G$ has a $k$-coloring in which every monochromatic component has at most $c$ vertices. We apply this result to show that $\chi_\star(\A_m)=\left\lfloor\frac{m}{2}\right\rfloor+1$, where $\A_m$ is the family of graphs $G$ with $\mad(G)\leq m$. This solves an open problem posed in Wood's survey and highlighted by Hendrey and Wood. 

As another application of our general partition result for $\mad$, we obtain a bound for relaxed colorings. Namely, for non-negative integers $d_1,\ldots,d_k$, every graph $G$ with \[\mad(G)<\sum_{i=1}^k \frac{2d_i+2}{d_i+2}\] admits a partition $V(G)=V_1\cup\ldots\cup V_k$ such that $\Delta(G[V_i])\leq d_i$ for each $i\in[k]$. In particular, this provides the first non-trivial bounds for $k\ge3$ with arbitrary $d_i$ and improves the previously known bound for $(d+1,d)$-colorings with $d \ge 2$.
\end{abstract}

\section{Introduction}
The \emph{maximum average degree} of $G$ is the maximum average degree taken over all (non-empty) subgraphs of $G$. Formally, let $V(G)$ denote the vertex set of $G$. For a subset of vertices $X \subseteq V(G)$, we let $G[X]$ denote the subgraph of $G$ induced by $X$, and let $e(G[X])$ denote its number of edges. The maximum average degree is defined as
\begin{equation*}
  \mad(G)
  :=
  \max_{\varnothing\ne X\subseteq V(G)}
  \frac{2e(G[X])}{|X|}.
\end{equation*}
We use the convention $\mad(H) = -\infty$ when $H$ has no vertices, so that the statements below also cover empty parts.

Our main result answers a question posed by Hendrey, Norin and Wood at the 2019 Barbados Graph Theory Workshop~\cite{CookSpirkl2019}. It was also conjectured by Nadara and Smulewicz~\cite[Conjecture~22]{NadaraSmulewicz2022}, who proved the special case $a \in \{1,2\}$. 

\begin{theorem}\label{thm:main}
For every graph $G$ and positive real numbers $a,b$, if $ \mad(G)<a+b,$ then there is a partition $V(G)=A\cup B$ such that $\mad(G[A])<a$ and $\mad(G[B])<b.$
\end{theorem}

Iterating \cref{thm:main} immediately gives the following multi-part form. 

\begin{corollary}\label{cor:multipart}
For all positive reals $c_1,\ldots,c_k$ and every graph $G$, if $\mad(G)<\sum_{i=1}^k c_i$, then there is a partition
$V(G) = V_1\cup\ldots\cup V_k$ such that $\mad(G[V_i])<c_i$ for every $i \in [k]$. 
\end{corollary}

\begin{proof}
The result follows by induction on $k$. Apply \cref{thm:main} with $a=c_1$ and $b=\sum_{i=2}^k c_i$, and then apply the induction hypothesis to the second part.
\end{proof}

We apply \cref{cor:multipart} to solve an open problem about clustered coloring.
A coloring has \emph{clustering} $c$ if each monochromatic component has at most $c$ vertices. 
The \emph{clustered chromatic number} of a graph class $\G$, denoted by $\chi_\star(\G)$, is the minimum integer $k$ for which there exists an integer $c$ such that every graph in $\G$ has a $k$-coloring with clustering $c$. 
For example, if $\mathcal{C}$ is the class of all cycles, then $\chi_\star(\mathcal{C})=2$, since every cycle has a $2$-coloring in which every monochromatic component has at most two vertices, while no uniform clustering bound is possible with one color.

Let $\A_m$ denote the family of graphs $G$ with $\mad(G)\leq m$.
It follows from the work of Havet and Sereni~\cite[Section~3.2]{HavetSereni2006} that $\chi_\star(\A_m)\geq \left\lfloor\frac{m}{2}\right\rfloor+1$.
The best previously known upper bound, $\chi_\star(\A_m)\leq\left\lfloor\frac{2}{3}m\right\rfloor+1$, was established by Hendrey and Wood~\cite{HendreyWood2019} in 2019. Determining the exact value of $\chi_\star(\A_m)$ was stated as an open problem in~\cite{MR4336225} and was again highlighted by Hendrey and Wood~\cite{HendreyWood2019} as an intriguing open problem. Here, we improve the upper bound and prove the following theorem, which closes this gap.

\begin{theorem}\label{thm:clustered}
    For every positive real $m$, it holds that $\chi_\star(\A_m)=\left\lfloor\frac{m}{2}\right\rfloor+1$.
\end{theorem}

\cref{cor:multipart} can also be applied to relaxed colorings.
For non-negative integers $d_1, d_2, \ldots, d_k$ we say that a graph $G$ is \emph{$(d_1, d_2, \ldots, d_k)$-colorable} if $V(G)$ can be partitioned into disjoint sets $V_1, V_2, \ldots, V_k$ such that, for each $i \in [k]$, the graph $G[V_i]$ has maximum degree at most $d_i$. For example, $K_3$ is $(1,0)$-colorable: two vertices may be placed in the first color class and the remaining vertex in the second, even though $K_3$ is not properly $2$-colorable.

We show (in \cref{lem:componentsize}) that a graph $H$ with $\mad(H) < \frac{2d+2}{d+2}$ has all components of order at most $d+1$ and so maximum degree at most $d$. 
Therefore, \cref{cor:multipart} and \cref{lem:componentsize} immediately imply the following theorem, which reproves several known results on relaxed colorings and improves some of them.

\begin{theorem}\label{thm:relaxedcoloring}
    For any non-negative integers $d_1, d_2, \ldots, d_k$ and every graph $G$ with 
    \[\mad(G) < \sum_{i=1}^k \frac{2d_i+2}{d_i+2}\]
    there is a partition $V(G) = V_1\cup\ldots\cup V_k$ such that, for every $i \in [k]$, all components of $G[V_i]$ have order at most $d_i+1$. In particular, $G$ is $(d_1, d_2, \ldots, d_k)$-colorable.
\end{theorem}

For $k=2$ and integers $d\ge j\ge0$, \cref{thm:relaxedcoloring} shows that every graph $G$ with
\[\mad(G) < 4 - \frac{2d+2j+8}{(d+2)(j+2)}\]
is $(d,j)$-colorable. 
For $d=j$ this is the bound of Havet and Sereni~\cite{HavetSereni2006}, and for $j=0$ it is the main result of Borodin, Ivanova, Montassier, Ochem and Raspaud in~\cite{d0-coloring}. In most other ranges stronger bounds are known: the sharp bound for $d\ge 2j+2$ of Borodin and Kostochka~\cite{dj-coloring}, the DP-coloring bounds for $j+2\le d\le 2j+1$ due to Jing, Kostochka, Ma, Sittitrai and Xu~\cite[Corollary~2.2]{JKMSX}, and improvements for particular small values of $j$. For $d=j+1\ge3$, however, \cref{thm:relaxedcoloring} improves the bound $\frac{2(2j^2+4j+1)}{j^2+3j+1}$ of~\cite[Corollary~2.2(4)]{JKMSX} to $4-\frac{4j+10}{(j+2)(j+3)}$.

\medskip
For $k \geq 3$, \cref{thm:relaxedcoloring} is the first non-trivial general result for $(d_1, d_2, \ldots, d_k)$-coloring with arbitrary $d_i$. It extends the non-list version of a theorem of Hendrey and Wood \cite[Theorem~2]{HendreyWood2019}, which covers the case $d_1 = \ldots = d_k$, as well as the results of~\cite{dd00-coloring,1100-coloring} for $d_i \in \{0,d\}$ discussed below. 
For instance, for any integers $a, b, d \geq 1$, every graph $G$ with 
\[\mad(G) < b + a\frac{2d+2}{d+2}\]
is $(d,d,\ldots,d, 0, 0, \ldots, 0)$-colorable with $a$ entries equal to $d$ and $b$ entries equal to $0$. This improves the bound $a+b+ \frac{da(a+1)}{(a+d+1)(a+1)+ab}$ of Dorbec, Kaiser, Montassier and Raspaud~\cite{dd00-coloring} and, for $d=1$, recovers the result of Kopreski and Yu~\cite{1100-coloring}. We note that this special case can also be derived by applying the result of Nadara and Smulewicz \cite{NadaraSmulewicz2022} $b$ times and then \cite[ Theorem~2]{HendreyWood2019}.

\section{Proof of the main result}\label{sec:main}

Before we proceed to the proof of \cref{thm:main}, we show the following lemma. 

\begin{lemma}\label{lem:labels}
For every graph $G$ and any real $q>\frac{1}{2}\mad(G)$, there exist positive real numbers $(c_v)_{v\in V(G)}$ such that for each $v\in V(G)$  
\[
    \sum_{u\in N(v)}\frac{c_u}{c_u+c_v}<q.
\]
\end{lemma}

\begin{proof}
We may assume that $G$ is connected, since otherwise it is enough to apply the argument below separately to each connected component of $G$ (for a component consisting of a single vertex $v$ we simply set $c_v = 1$), using the fact that $\mad(H) \leq \mad(G)$ for every subgraph $H$ of $G$.

For a graph $G$, define a network $\mathcal F(G)$ as follows. Its nodes consist of a source $s$, a sink $t$, one node $x_{uv}$ for each edge $uv\in E(G)$, and one node for each vertex of $G$. For each edge $uv\in E(G)$, there is an arc $sx_{uv}$ of capacity $1$ and arcs $x_{uv}u$ and $x_{uv}v$ of infinite capacity. Finally, for each $v\in V(G)$, there is an arc $vt$ of capacity $\frac{1}{2}\mad(G)$. See \cref{fig:flow-network} for an example.

\begin{figure}[ht]
    \centering
    \includegraphics[width=0.9\textwidth]{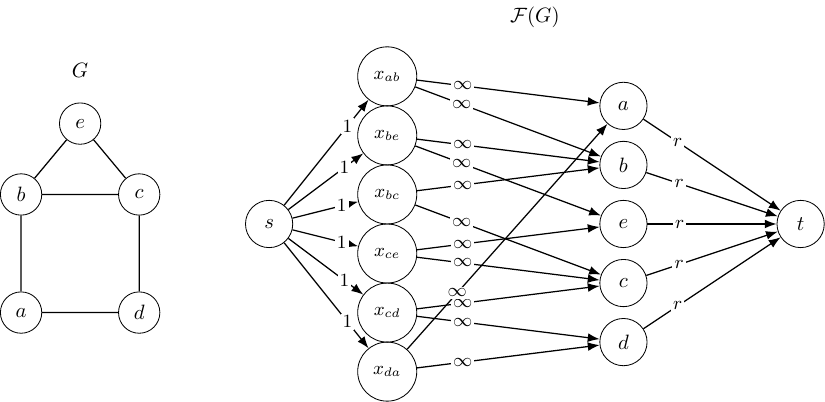}
    \caption{A graph $G$ and the network $\mathcal F(G)$, where $r= \frac{1}{2}\mad(G)$.}
    \label{fig:flow-network}
\end{figure}

Nadara and Smulewicz \cite{NadaraSmulewicz2022} proved that the network $\mathcal F(G)$ has a maximum $s$--$t$ flow of size $e(G)$.
Note that this also follows directly from the max-flow min-cut theorem: if $X$ is the set of vertices of $G$ on the source side of an $s$--$t$ cut of finite capacity, then every node $x_{uv}$ on the source side satisfies $u, v \in X$, so the capacity of the cut is at least $e(G) - e(G[X]) + \frac12 \mad(G)|X| \geq e(G)$. A maximum flow of size $e(G)$ saturates every arc leaving $s$, so letting $p'_{uv}$ be the flow along the arc $x_{uv}v$, we obtain numbers $p'_{uv}\in[0,1]$ such that $p'_{uv}+p'_{vu}=1$  for each edge $uv\in E(G)$, and 
\[\sum_{u\in N(v)}p'_{uv}\leq \frac{1}{2}\mad(G)\]
for each vertex $v\in V(G)$.

We now perturb the constants $p'_{uv}$ slightly so that none of them is equal to $0$ or $1$. 
Consider the assignment $p''_{uv}=\frac{1}{2}$ for each $uv\in E(G)$, and for $\varepsilon\in(0,1)$ let $p_{uv}=(1-\varepsilon)p'_{uv}+\varepsilon p''_{uv}$ for each $uv\in E(G)$. 
Then $p_{uv}\in(0,1)$ and $p_{uv}+p_{vu}=1$ for each $uv\in E(G)$, while for each $v\in V(G)$ 
\[\sum_{u\in N(v)}p_{uv}\leq(1-\varepsilon)\frac{1}{2}\mad(G)+\varepsilon\frac{d(v)}{2}.\]
As the graph is finite
and $q>\frac{1}{2}\mad(G)$, by taking $\varepsilon>0$ sufficiently small we may assume that for each $v\in V(G)$ 
\begin{equation}\label{eq:BTmodel}
\sum_{u\in N(v)}p_{uv}<q.
\end{equation}

Such an assignment can be viewed as a Bradley–Terry model for paired comparisons in which the players are the vertices of $G$, the compared pairs are the edges of $G$, and, for every $uv \in E(G)$, player $u$ beats player $v$ with probability $p_{uv}$. This model was first studied by Zermelo \cite{Zermelo1929} in~1929, and independently by Bradley and Terry \cite{BradleyTerry1952} in 1952. As $p_{uv}\in(0,1)$ for every $uv\in E(G)$ and $G$ is connected, for every partition of $V(G)$ into two non-empty parts some player of each part beats some player of the other part with positive probability. Hence, by the result of Ford \cite{Ford1957}, one can assign a value $c_v>0$ to each $v \in V(G)$ such that 
\[\sum_{u \in N(v)} p_{uv} = \sum_{u \in N(v)} \frac{c_u}{c_u+c_v}\]
for every $v \in V(G)$. 
For an explicit proof see, e.g., \cite[Corollary~1]{SchnakenbergPenn2014}. 
By \eqref{eq:BTmodel}, the constants $c_v$ satisfy the conditions of the lemma.
\end{proof}

The main idea in the proof of \cref{thm:main} is to define a suitable potential function and show that the partition minimizing this function has the properties we wish for.

\begin{proof}[Proof of \cref{thm:main}]
Applying \cref{lem:labels} with $q = \frac{1}{2}(a+b)$, which
satisfies $q > \frac12\mad(G)$ as $\mad(G) < a + b$, we obtain positive real numbers $c_v$, for each $v\in V(G)$, such that 
\begin{equation}\label{eq:label-bound}
    \sum_{u\in N(v)}\frac{c_u}{c_u+c_v} < \frac{a+b}{2}
\end{equation}
for every $v\in V(G)$. Let 
\[w(uv) = \frac{c_uc_v}{c_u+c_v}\]
for every $uv \in E(G)$.

Define the \emph{potential} of a partition $V(G) = A \cup B$ by 
\[
\Phi({A,B}) \coloneqq \sum_{uv\in E(G[A])\cup E(G[B])} w(uv) - \frac{a}{2}\sum_{u \in A} c_u - \frac{b}{2}\sum_{v\in B} c_v.
\]

Among all partitions $V(G)=A\cup B$, choose one that minimizes $\Phi({A,B})$.
For any $v \in A$, consider the partition obtained by moving $v$ from $A$ to $B$, namely
$A'\coloneqq A\setminus\{v\}$, $B'\coloneqq B\cup\{v\}.$
By the minimality of \(\Phi(A,B)\), we have
\[
    \Phi(A,B)\leq \Phi(A',B').
\]
Only the terms corresponding to edges incident with \(v\), together with
the term involving \(c_v\), change when \(v\) is moved. Consequently,
\[
    \sum_{u \in N(v) \cap A} w(uv)-\frac{a}{2}c_v
    \leq
    \sum_{u \in N(v) \cap B} w(uv)-\frac{b}{2}c_v.
\]
Dividing by $c_v > 0$ and using the definition of $w(uv)$, we obtain
\[
    \sum_{u \in N(v) \cap A}
        \frac{c_u}{c_u+c_v}-\frac{a}{2}
    \leq
    \sum_{u \in N(v) \cap B}
        \frac{c_u}{c_u+c_v}-\frac{b}{2}.
\]
Equivalently,
\[
    \sum_{u \in N(v) \cap A} \frac{c_u}{c_u+c_v}- \sum_{u \in N(v) \cap B} \frac{c_u}{c_u+c_v}
    \leq
    \frac{a-b}{2}.
\]
On the other hand, by \eqref{eq:label-bound},
\[
    \sum_{u \in N(v) \cap A} \frac{c_u}{c_u+c_v} + \sum_{u \in N(v) \cap B} \frac{c_u}{c_u+c_v}
    =
    \sum_{u \in N(v)} \frac{c_u}{c_u+c_v}
    < \frac{a+b}{2}.
\]
Adding the preceding two inequalities gives
\[
    2\sum_{u \in N(v) \cap A}
        \frac{c_u}{c_u+c_v}
    < a
\]
for every $v \in A$.
The case \(v\in B\) is symmetric, so we also have 
\[
    2\sum_{u \in N(v) \cap B}
        \frac{c_u}{c_u+c_v}
    < b
\]
for every $v \in B$.
Using the above inequalities, we show that $\mad(G[A])<a$ and $\mad(G[B])<b$.
For any non-empty set $X \subseteq A$, we have 
\begin{align*}
    2e(G[X])
    &= 2\sum_{uv\in E(G[X])}
        \left(\frac{c_u}{c_u+c_v}+\frac{c_v}{c_u+c_v}\right) \\
    &= 2\sum_{v\in X}
        \sum_{u \in N(v) \cap X}
        \frac{c_u}{c_u+c_v} \\
    &\leq \sum_{v\in X}
        2\sum_{u \in N(v) \cap A}
        \frac{c_u}{c_u+c_v} \\
    &< \sum_{v\in X} a
     = a|X|.
\end{align*}
Therefore, $\mad(G[A])<a$. 
Similarly, it follows that for any non-empty set $Y \subseteq B$, $2e(G[Y]) < b|Y|$ and $\mad(G[B])<b$. 
Thus the partition \(V(G)=A \cup B\) has the desired properties.
\end{proof}

\section{Applications in graph colorings}\label{sec:coloring}

We start the section with a straightforward lemma, which together with \cref{cor:multipart} immediately implies \cref{thm:relaxedcoloring}.

\begin{lemma}\label{lem:componentsize}
    For any positive integer $c$, every graph $H$ with $\mad(H)<\frac{2c}{c+1}$ has all components of order at most $c$.
\end{lemma}

\begin{proof}
    Suppose otherwise. Then some component of $H$ has at least $c+1$ vertices, and hence $H$ contains a tree $T$ on exactly $c+1$ vertices. Then \[\mad(H)\geq\mad(T)=\frac{2c}{c+1},\] a contradiction. 
\end{proof}

Since $\chi_\star(\A_m)\geq\left\lfloor\frac{m}{2}\right\rfloor+1$ by the result of Havet and Sereni~\cite{HavetSereni2006}, in order to show \cref{thm:clustered} 
it is enough to prove the following statement. 

\begin{theorem}
    For every positive real $m$, let $k=\left\lfloor\frac{m}{2}\right\rfloor+1$ and $c=\left\lceil\frac{2k}{2k-m}\right\rceil$. Every graph $G$ with $\mad(G)\leq m$ has a $k$-coloring with clustering $c$.
\end{theorem}

\begin{proof}
    Note that since $2k>m$, we have $(2k-m)^2>0$ and $4k-m>0$, hence $4k^2 > m(4k-m)$, that is $\frac{4k}{4k-m}\cdot k>m$. Thus, $\mad(G)<\frac{4k}{4k-m}\cdot k$. By \cref{cor:multipart} there exists a partition $V_1\cup\ldots\cup V_k$ such that $\mad(G[V_i])<\frac{4k}{4k-m}$ for every $i\in[k]$.

    Since $\frac{2c}{c+1}$ is an increasing function of $c$ and $c \ge \frac{2k}{2k-m}$, we have 
    \[\frac{2c}{c+1} \ge \frac{\frac{4k}{2k-m}}{\frac{2k}{2k-m}+1} = \frac{4k}{4k-m}.\] 
    
Hence, $\mad(G[V_i])<\frac{2c}{c+1}$ for every $i\in[k]$. By \cref{lem:componentsize}, each component of $G[V_i]$, for every $i\in[k]$, has at most $c$ vertices, as needed.
\end{proof}

\section{Concluding remarks}

Note that \cref{thm:main} is tight: for positive integers $a$, $b$ the assumption $\mad(G) < a+b$ cannot be weakened to $\mad(G) \leq a+b$. Indeed, by considering $G$ to be the complete graph on $a+b+1$ vertices, we have $\mad(G)=a+b$. But for any vertex partition $V(G)=A\cup B$, either $|A|\geq a+1$ or $|B|\geq b+1$, and since $G[A]$ and $G[B]$ are complete graphs, this implies that either $\mad(G[A])\geq a$ or $\mad(G[B])\geq b$, respectively.

\section*{Acknowledgements}
We thank the organizers of the Emléktábla Workshop, where this work was initiated, as well as Laurentiu Ioan Ploscaru for suggesting the problem that inspired it.

\paragraph{Generative AI statement.} The potential function used in the proof of \cref{thm:main} was initially suggested by ChatGPT-5.6 Sol. However, the proof generated by the model relied on a technically involved lemma whose validity was difficult to verify. The authors subsequently developed the proof presented in the final manuscript and take full responsibility for the correctness and content of the paper.

\bibliographystyle{abbrv}
\bibliography{MAD}
\end{document}